%% file: main.tex
\input{preamble/packages}

\input{preamble/macros}

\usepackage[urldate=long, backend=biber, style = alphabetic, sorting = nyt]{biblatex}
\title[The Hasse principle for block-diagonal forms]{On the Hasse principle for a block-diagonal form and a form of lower degree}

\author{Anna Theorin Johansson}
\address{Department of Mathematical Sciences, Chalmers University of Technology and the University of Gothenburg, SE-412 96 Gothenburg, Sweden}
\email{annatheo@chalmers.se} 
\date{}

\begin{document}

\begin{abstract}
    We investigate the Hasse principle for systems consisting of one form of degree $k$ which decomposes as a sum of $n$-ary forms, and one general form of degree $d < k$. By exploiting the block structure of the first form before applying Weyl differencing, we refine the method of Brandes and Parsell \cite{brandes_hasse_2021} in this setting; in particular, for a sum of binary quartic forms and a quadratic form, the requirement $s > 516$ is relaxed to $s > 96$, which also improves on the bound $s > 100$ available for general forms.
\end{abstract}

\maketitle

\section{Introduction}

\subsection{Motivation}
Suppose we are given a pair of forms $F, G \in \Z[x_1, \dots, x_s]$ of degrees $k$ and $d$, respectively. We are interested in the number $N_{F, G}(X)$ of integer points $\xu \in [-X, X]^s$ satisfying $F(\xu) = G(\xu) = 0$. For $s$ sufficiently large and for suitably non-singular forms, we expect this to be of the form 
\begin{equation}\label{eq:asymp_temp}
    N_{F, G}(X) = X^{s-k-d} \of{\Constant + O\of{X^{-\nu}}},
\end{equation}
where $\Constant \geq 0$ is a constant and $\nu > 0$ is a small number, both dependent at most on the forms $F$ and $G$. Let $\mathfrak{s}_0 = \mathfrak{s}_0(k, d)$ denote the smallest integer such that \cref{eq:asymp_temp} holds for all $s > \mathfrak{s}_0$ and all suitably non-singular forms $F$ and $G$ in the class under consideration; note that $\mathfrak{s}_0$ depends on this class. The objective is to find upper bounds for $\mathfrak{s}_0$, and in particular to understand how these improve when one of the forms has additional structure.

\subsection{Previous work}
If both $F$ and $G$ are diagonal, Wooley's efficient congruencing method gives
\begin{equation*}
    \mathfrak{s}_0(k, d) \leq k(k+1)
\end{equation*} 
for any two distinct degrees $k > d$, with better bounds if $d$ is much smaller than $k$ \cite[\S 14]{wooley_nested_2019}. In the case $k = 3$ and $d = 2$, one can achieve $\mathfrak{s}_0(3, 2) \leq \frac{32}{3}$, so that $s = 11$ suffices \cite{wooley_rational_2015}. The underlying mean value estimate was recently sharpened to an essentially optimal form \cite[Theorem 1.1]{wooley_main_2026}, and as remarked in \S 7 of the same paper, this improvement should allow for a further reduction in the number of variables required in this case. 

On the other hand, without any assumption of diagonality, one must resort to the work of Browning and Heath-Brown \cite{browning_forms_2017}, which treats systems of forms of differing degrees in full generality. Their bound grows exponentially in both $k$ and $d$; to be precise, \cite[Corollary 1.8]{browning_forms_2017} gives
\begin{equation}\label{eq:s_0_BHB}
    \mathfrak{s}_0(k, d) \leq (2+d)(k-1)2^{k-1} + 2^{d-1} d. 
\end{equation}
With $k = 3$ and $d = 2$, this corresponds to $\mathfrak{s}_0(3, 2) \leq 36$, which was sharpened by Browning, Dietmann and Heath-Brown \cite[Theorem 1.3]{browning_rational_2015} to
\begin{equation}\label{eq:s_0_32_BDHB}
    \mathfrak{s}_0(3, 2) \leq 28.
\end{equation}

In the case where $F$ is diagonal but no specific shape restriction is put on $G$, Brandes and Parsell \cite{brandes_hasse_2021} used the diagonal structure of $F$ to replace the exponential growth in $k$ in \cref{eq:s_0_BHB} by a quadratic one. With $F$ diagonal of degree $k$ and $G$ of degree $d < k$, \cite[Theorem 1.1]{brandes_hasse_2021} gives
\begin{equation}\label{eq:s_0_kd_BP}
    \mathfrak{s}_0(k, d) \leq 
    \begin{cases}
        2^k(d+1) & \quad \text{if } d+1 \leq k \leq d+4, \\
        2^d(26 + 32d) & \quad \text{if } k = d + 5, \\
        2^d\left[(2d+1)k^2 - L_d(k) \right] & \quad \text{if } k \geq d + 6,
    \end{cases} 
\end{equation}
where $L_d(k)$ is an explicit function of $d$ and $k$; see the original paper and the forthcoming corrigendum for its precise form. In the case of $k = 3$ and $d = 2$, this gives
\begin{equation}\label{eq:s_0_32_BP}
    \mathfrak{s}_0(3, 2) \leq 24.
\end{equation}
In the case $d = k-1$, this was recently improved to $\mathfrak{s}_0(k, k-1) \leq 2^{k-1} (2k-1)$ in \cite{theorin_johansson_hasse_2026}, so that in particular
\begin{equation}\label{eq:s_0_32_ATJ}
    \mathfrak{s}_0(3, 2) \leq 20.
\end{equation}

A natural next question is what happens if $F$ is not diagonal, but still has some additive structure. To make this precise, let $n \in \N$ with $n \mid s$, write $s = nu$, and for a vector $\xu \in \Z^s$, let $\xu = (\x^{(1)}, \dots, \x^{(u)})$ with each $\x^{(i)} \in \Z^n$. We say that a form $F \in \Z[x_1, \dots, x_s]$ of degree $k$ \emph{decomposes as a sum of $n$-ary forms} if, after a permutation of the variables,
\begin{equation}\label{eq:F_psi}
    F(\xu) = \psi_1(\x^{(1)}) + \dots + \psi_u(\x^{(u)}),
\end{equation}
for some forms $\psi_1, \dots, \psi_u \in \Z[x_1, \dots, x_n]$ of degree $k$. For $n = 1$, this is the definition of a diagonal form. Such forms allow for separation of variables just as in the diagonal case, but leave some freedom in the shape of the summands. For $n = 2$, Weyl-type estimates and mean values for exponential sums associated with binary forms were developed by Wooley \cite{wooley_weyls_1999}. In the setting of the present paper, Brandes and Parsell \cite{brandes_hasse_2021} showed that their methods extend to forms $F$ which decompose as a sum of $n$-ary forms; we recall this in \cref{thm:bp_thm} below.

In order to state our result, we introduce the local solution densities associated with the system $F(\xu) = G(\xu) = 0$. For $q \in \N$, let 
\begin{equation*}
    \Gamma(q) \eqdef \# \set{\xu \in (\Z/q\Z)^s : F(\xu) \equiv G(\xu) \equiv 0 \Mod q},
\end{equation*}
and for each prime $p$, set
\begin{equation*}
    \chi_p \eqdef \lim_{h \to \infty} \of{p^{h}}^{2-s} \Gamma(p^h).
\end{equation*}
Next, define for $(\gamma, \delta) \in \R^2$ the integral 
\begin{equation}\label{eq:J_def}
    J(\gamma, \delta) \eqdef \int_{[-1, 1]^s} e \of{\gamma F(\bm{\xi}) + \delta G(\bm{\xi})} \diff \bm{\xi},
\end{equation}
and let
\begin{equation*}
    \chi_{\infty} \eqdef \lim_{R \to \infty} \int_{[-R, R]^2} J(\gamma, \delta) \diff \gamma \diff \delta.
\end{equation*}
Define
\begin{equation*}
    \Constant_{F, G} \eqdef \chi_{\infty} \prod_p \chi_p;
\end{equation*}
we will show in \cref{sec:major} that these limits do indeed exist. Heuristically, $\chi_p$ can be interpreted as the density of solutions of $F(\xu) = G(\xu) = 0$ in $\Z_p^s$, and $\chi_{\infty}$ in turn as the density of real solutions in the box $[-1, 1]^s$.

\begin{definition}\label{def:sigma}
    For $j \geq 1$, let 
    \begin{align*}
        \tau_0(j) \eqdef \min\set{2^{j-1}, \frac{1}{2} j(j-1) + \floor{\sqrt{2j+2}}}, && \sigma_0(j) \eqdef \begin{cases}
            2^{j-1} & \quad \text{if } 1 \leq j \leq 5, \\
            j(j-1) & \quad \text{if } j \geq 6.
        \end{cases}
    \end{align*}
\end{definition}

\begin{theorem}[{\cite[Theorem 1.7]{brandes_hasse_2021}}]\label{thm:bp_thm}
    Let $k > d \geq 2$ and let $n \in \N$ with $n \mid s$. Suppose that $F \in \Z[x_1, \dots, x_s]$ is a non-singular form of degree $k$ which decomposes as a sum of $n$-ary forms, and that $G \in \Z[x_1, \dots, x_s]$ is a non-singular form of degree $d$. Let $\sigma_0$ be as in \cref{def:sigma}, and suppose that $s > s^{\mathrm{BP}}$, where
    \begin{equation}
        s^{\mathrm{BP}} = s^{\mathrm{BP}}(n, k, d) \eqdef 2^{d-1} d + 2^{d+1} (d + 2) n \sigma_0(k). \label{eq:req_old_1}
    \end{equation}
    Then there is a $\nu > 0$ such that
    \begin{equation}\label{eq:asymp}
        N_{F, G}(X) = X^{s-k-d} \of{\Constant_{F, G} + O\of{X^{-\nu}}}.
    \end{equation}
\end{theorem}

\subsection{Results}
The goal of this paper is to streamline the proof of the above to improve the requirement \cref{eq:req_old_1}. The following is our main result.
\begin{theorem}\label{thm:main}
    Let $k > d \geq 2$ and let $n \in \N$ with $n \mid s$. Suppose that $F \in \Z[x_1, \dots, x_s]$ is a non-singular form of degree $k$ which decomposes as a sum of $n$-ary forms, and that $G \in \Z[x_1, \dots, x_s]$ is a non-singular form of degree $d$. Let $\tau_0$ and $\sigma_0$ be as in \cref{def:sigma}, and suppose that $s > s_*$, where
    \begin{align}
        s_* = s_*(n, k, d) &\eqdef 
        \begin{cases}
            (2d+1) 2^d n & \quad \text{if $k - d = 1$}, \\
            2^{d+1} n \of{\tau_0(k-d) + d \sigma_0(k-d)} & \quad \text{if } k - d \geq 2.
        \end{cases}
        \label{eq:req_new_1} 
    \end{align}
    Then there is a $\nu > 0$ such that the asymptotic formula \cref{eq:asymp} holds, with the constant $\Constant_{F, G}$ a product of local solution densities. 
\end{theorem}

\begin{remark}\label{rmk:s2_s3}
    The proof of \cref{thm:main} also requires $s > s_2$ and $s > s_3$, where
    \begin{align}
        s_2 = s_2(n, k, d) &\eqdef 2^{d-1} d + 2^{d} n (d + 2) \sigma_0(k-d), \label{eq:req_new_2} \\
        s_3 = s_3(n, k, d) &\eqdef 2^{d} (d-1) + 3 \cdot 2^{d} n\sigma_0(k-d).\label{eq:req_new_3}
    \end{align}
    Both conditions are implied by $s > s_*$. Writing $s_1 \eqdef s_*$, the numbering of $s_1$, $s_2$ and $s_3$ follows the order in which the corresponding conditions appear in \cite{browning_forms_2017}.
\end{remark}

\begin{remark}\label{rmk:range}
    If $n = 1$, the form $F$ is diagonal, and \cref{thm:main} recovers the bound of \cite[Theorem 1.1]{brandes_hasse_2021} (as corrected in the corrigendum \cite{brandes_parsell_corrigendum}) if $k - d \geq 2$, and that of \cite{theorin_johansson_hasse_2026} if $k = d+1$; the new cases are therefore those with $n \geq 2$. In these cases, \cref{thm:main} always improves on \cref{thm:bp_thm}. Note however that the bound \cref{eq:s_0_BHB} for general forms does not depend on $n$, whereas $s_*$ grows linearly in $n$, which limits the chance of improvement for large $n$.
\end{remark}

In particular, for the case of one sum of binary quartic forms and one quadratic form, the requirement is improved as follows.
\begin{corollary}
    Let $F, G \in \Z[x_1, \dots, x_s]$ be a pair of non-singular forms, where $F$ is quartic and decomposes as a sum of binary forms, and $G$ is quadratic. Suppose that
    \begin{equation*}
        s > 96.
    \end{equation*}
    Then there is a $\nu > 0$ such that
    \begin{equation*}
        N_{F, G}(X) = X^{s-6} \of{\Constant_{F, G} + O\of{X^{-\nu}}},
    \end{equation*}
    and $\Constant_{F, G} > 0$ whenever the system $F(\xu) = G(\xu) = 0$ has a non-singular solution in $\R^s$ and in $\Q_p^s$ for every prime $p$. 
\end{corollary}
For comparison, in this specific situation, \cref{thm:bp_thm} requires $s > 516$ and the bound \cref{eq:s_0_BHB} for general forms is $s > 100$.

\subsection{Main idea of the proof}
The main improvement arises from exploiting the block structure of $F$ before applying Weyl differencing. After a unimodular change of variables in each $n$-variable block, we may arrange for each summand to have a non-zero coefficient of a pure $k$th power, before we fix the remaining $n-1$ variables in every block. This reduces the relevant inner exponential sum to a sum in $u = s/n$ one-dimensional variables.

To these variables, we apply $d$ Weyl differencing steps, so that, since $G$ has degree $d$, its contribution vanishes upon taking absolute values. The contribution from $F$, on the other hand, separates into one-dimensional exponential sums whose phases are polynomials of degree $k-d$. This allows us to use the sharper one-dimensional estimates involving $\tau_0(k-d)$ and $\sigma_0(k-d)$. In comparison with \cite[Theorem 1.7]{brandes_hasse_2021}, this replaces the occurrences of $\sigma_0(k)$ with $\sigma_0(k-d)$ and, in $s_2$ and $s_3$, each $2^{d+1}$ with $2^d$.

\subsection{Outline of the paper}
In \cref{sec:weyl}, we define the one-dimensional major arcs and collect the one-dimensional Weyl and mean value estimates needed. In \cref{sec:minor}, we express the counting function $N_{F, G}(X)$ as an integral of an exponential sum $T(\alpha, \beta)$ and treat the minor arc contribution, first for the one-dimensional minor arcs in $\alpha$, where the block structure of $F$ is used, and then for the two-dimensional subdivision in $(\alpha, \beta)$. In \cref{sec:major}, we evaluate the remaining major arc contribution and establish the convergence of the singular series and singular integral, so that we can complete the proof of \cref{thm:main} in the final section. 

\subsection{Notation}
We will use the following conventions throughout:
\begin{itemize}
    \item For a vector $\x \in \R^m$, define $\abs{\x} \eqdef \norm{\x}_{\infty} = \max_{i \in \set{1, \dots, m}} \abs{x_i}$.
    \item For $\alpha \in \R$, let $e(\alpha) = e^{2 \pi i \alpha}$ and $e_q(\alpha) = e(\alpha/q)$ whenever $q \in \N$.
    \item Let $\T = \R/\Z$, and write $\operatorname{vol}$ for the Lebesgue measure on $\T$ and on $\T^2$.
    \item For $\beta\in \R$, write $\norm{\beta} \eqdef \min_{b \in \Z} \abs{\beta - b}$ for the distance from $\beta$ to the nearest integer. Since $\norm{\beta}$ depends only on the residue class of $\beta$ modulo $1$, this also defines $\norm{\alpha}$ for $\alpha \in \T$, which allows us to use $\norm{\cdot}$ on both $\R$ and $\T$.
    \item If $f \colon \T^m\to \C$ is integrable, define
    \begin{equation*}
        \oint f(\bm{\xi}) \diff \bm{\xi} \eqdef \int_{\T^m} f(\bm{\xi}) \diff \bm{\xi},
    \end{equation*}
    where the dimension $m$ is always clear from the context.
    \item When $\eps$ appears in a statement, it is asserted that the statement holds for every $\eps > 0$, with implied constants that may depend on $\eps$.
    \item Whenever $X$ appears in a statement, it is asserted that the statement holds for all sufficiently large $X$, that is, for all $X \geq X_0$, where $X_0$ may depend on every parameter other than $X$ (including $\eps$).
    \item The implied constants in the Landau and Vinogradov notations are always allowed to depend on every parameter but $X$, unless otherwise specified.
\end{itemize}

\subsection{Acknowledgements}
The author is grateful to Simon Rydin Myerson for proposing this problem, and to Julia Brandes and Simon Rydin Myerson for helpful comments on an earlier draft.

\section{Weyl estimates}\label{sec:weyl}

For $j, q \in \N$, $a \in \Z$ and $\theta \in (0, 1]$, define
\begin{equation*}
    \Ma_j(q, a; \theta) \eqdef \set{\alpha \in \T: \abs{\alpha - \frac{a}{q}} \leq q^{-1} X^{-j + \theta}}.
\end{equation*}

\begin{definition}\label{def:major_arcs}
    For $j \in \N$ and $\theta \in (0, 1]$, let $\Ma_j(\theta)$ be the union of the arcs $\Ma_j(q, a; \theta)$ with $(a, q) = 1$ and $1 \leq a \leq q \leq X^{\theta}$, and set $\mi_j(\theta) \eqdef \T \setminus \Ma_j(\theta)$. For ease of notation, we write $\Ma(\theta) \eqdef \Ma_k(\theta)$ and $\mi(\theta) \eqdef \mi_k(\theta)$, where $k$ is the degree of $F$.
\end{definition}

Since $k \geq 3$, the arcs comprising $\Ma(\theta)$ are pairwise disjoint for all $\theta \in (0, 1]$, and
\begin{equation}\label{eq:major_volume}
    \operatorname{vol} \Ma(\theta) \ll \sum_{q \leq X^{\theta}} \sum_{\substack{a = 1 \\ (a, q) = 1}}^q \frac{X^{-k + \theta}}{q} \ll X^{-k + 2 \theta}. 
\end{equation}

Recall the definitions of $\tau_0(j)$ and $\sigma_0(j)$ from \cref{def:sigma}. Let $j, m \in \N$, let $\varphi \in \Z[x_1, \dots, x_m]$ be a polynomial of degree $j$ in $m$ variables and let $\mathcal{D} \subseteq \R^m$ be a bounded convex set.
\begin{definition}
    Define
    \begin{equation*}
        f_{j, m}(\alpha) = f_{j, m}(\alpha; \varphi, \mathcal{D}) \eqdef \sum_{\x \in \mathcal{D} \cap \Z^m} e \of{\alpha \varphi(\x)}.
    \end{equation*}
    When $\varphi$ and $\mathcal{D}$ are clear from context, we simply write $f_{j, m}(\alpha)$. 
\end{definition}

The following assumption on the coefficients of $\varphi$ will be necessary in what follows. 
\begin{hypothesis}\label{hyp}
    Suppose that $\varphi$ has the property that for each $1 \leq i \leq j$, the degree-$i$ part of $\varphi$ has coefficients of size $\ll X^{j-i}$. 
\end{hypothesis}

\begin{theorem}\label{thm:2.1}
    Let $\varphi \in \Z[x_1, \dots, x_m]$ be of degree $j \geq 1$, and suppose that $\varphi$ satisfies \cref{hyp}. Let $\mathcal{D}$ be the image of $[-X, X]^m$ under a unimodular linear transform, with the property that $\max\set{\abs{\x}: \x \in \mathcal{D}} \leq \tau X$ for some constant $\tau$. 
    \begin{enumerate}\label{={(\roman*)}}
        \item Suppose that $u \geq \tau_0(j)$. Then
        \begin{equation*}
            \oint \abs{f_{j, m}(\alpha)}^{2u} \diff \alpha \ll X^{2 m u - j + \eps}.
        \end{equation*}
        \item Let $\alpha \in \T$ and let $q \in \N$ satisfy $\norm{\alpha q} \leq q^{-1}$. Then
        \begin{equation*}
            \abs{f_{j, m}(\alpha)} \ll X^{m + \eps} \of{\frac{1}{q} + \frac{1}{X} + \frac{q}{X^j}}^{\frac{1}{\sigma_0(j)}}.
        \end{equation*}
    \end{enumerate}
    In both statements, the implied constants are allowed to depend on $\tau$.
\end{theorem}
\begin{proof}
    This follows the proof of \cite[Theorem 2.1]{brandes_hasse_2021}, with some modifications to be presented in \cite{brandes_parsell_corrigendum}. The main point is that the reference to \cite[Corollary 14.8]{wooley_nested_2019} should be replaced by a reference to the first equation in the proof of said corollary. Thanks to \cref{hyp}, the equation holds uniformly in the coefficients of $\varphi$, and the constant $A$ occurring in the proof of \cite[Theorem 2.1]{brandes_hasse_2021} is well-defined.
\end{proof}

\begin{lemma}\label{lemma:2.2}
    Let $j, r \in \N$, let $C \geq 1$ and suppose that $\alpha \in \mi_{j +r}(\theta)$ for some $\theta \in (0, 1]$. For each $1 \leq h \leq CX^r$, let $\varphi_{h} \in \Z[x]$ be a polynomial of degree $j$, let $\mathcal{D}_{h} \subseteq [-CX, CX]$ be an interval, and suppose that the $\varphi_h$ satisfy \cref{hyp} with implied constants independent of $h$. Then 
    \begin{equation*}
        \sum_{h = 1}^{X^r} \abs{f_{j, 1}(h \alpha; \varphi_h, \mathcal{D}_h)} \ll X^{r + 1 - \frac{\theta}{\sigma_0(j)} + \eps},
    \end{equation*}
    where the implied constant depends on the $\varphi_h$ only through the constants in \cref{hyp}.
\end{lemma}
\begin{proof}
    Given the alterations to the proof of \cite[Theorem 2.1]{brandes_hasse_2021}, this is \cite[Lemma 2.2]{brandes_hasse_2021} with a slight correction to appear in the aforementioned corrigendum. In the proof, each summand is estimated individually, by bounds which hold uniformly in $\varphi$ and $\mathcal{D}$ under the above hypotheses.
\end{proof}

\begin{lemma}\label{lemma:2.3}
    Let $j, r \in \N$, let $C \geq 1$, and suppose that $\alpha \in \mi_{j +r}(\theta)$ for some $\theta \in (0, 1]$. For each $\h = (h_1, \dots, h_r) \in \Z^r$, let $p_{\h} \in \Z[x]$ be a polynomial of degree $j$ satisfying \cref{hyp} with implied constants independent of $\h$, and let $\mathcal{D}_{\h} \subseteq [-CX, CX]$ be an interval. Then
    \begin{equation*}
        \sum_{\substack{\abs{h_1}, \dots, \abs{h_r} \leq CX \\ h_1 \cdots h_r \neq 0}} \abs{f_{j, 1} (h_1 \cdots h_r \alpha; p_{\h}, \mathcal{D}_{\h})} \ll X^{r + 1 - \frac{\theta}{\sigma_0(j)} + \eps },
    \end{equation*}
    where the implied constant depends on the $p_{\h}$ only through the constants in \cref{hyp}.
\end{lemma}

\begin{proof}
    The hypothesis made on $p_{\h}$ means that this follows from \cref{lemma:2.2} with a standard divisor estimate. Note that both the polynomial $p_{\h}$ and the interval $\mathcal{D}_{\h}$ here are (crucially) allowed to depend on $\h$. 
\end{proof}

\section{Minor arc contribution}\label{sec:minor}
For $(\alpha, \beta) \in \T^2$, define
\begin{equation*}
    T(\alpha, \beta) \eqdef \sum_{\xu \in [-X, X]^s \cap \Z^s} e\of{\alpha F(\xu) + \beta G(\xu)}.
\end{equation*}
Since 
\begin{equation*}
    \oint e (\alpha m ) \diff \alpha = \begin{cases}
        1 & \quad \text{if } m = 0 \\
        0 & \quad \text{if } m \in \Z \setminus \set{0},
    \end{cases}
\end{equation*}
we have 
\begin{equation*}
    N_{F, G}(X) = \oint \oint T(\alpha, \beta) \diff \alpha \diff \beta.
\end{equation*}

We want to show that if all of
\begin{align}
    s &> 
    \begin{cases}
        (2d+1) 2^d n \sigma & \quad \text{if } k - d = 1, \\
        2^{d+1}n(\tau + d \sigma) & \quad \text{if } k - d \geq 2, 
    \end{cases} 
    \label{eq:req_temp_1} \\
    s &> 2^{d-1} d + 2^{d} n (d + 2) \sigma, \label{eq:req_temp_2} \\
    s &> 2^{d} (d-1) + 3 \cdot 2^{d} n\sigma \label{eq:req_temp_3}
\end{align}
are satisfied, then for some $\nu > 0$, the asymptotic formula 
\begin{equation*}
    N_{F, G}(X) = X^{s - k - d} \chi_{\infty} \prod_p \chi_p + O(X^{s-k-d-\nu})
\end{equation*}
holds, where $\chi_{\infty}$ and $\chi_p$ correspond to local solution densities, as outlined in the introduction. Let
\begin{equation*}
    u \eqdef \frac{s}{n}, \quad \tau \eqdef \tau_0(k-d), \quad \sigma \eqdef \sigma_0(k-d), \quad t \eqdef \frac{s}{2^d n}.
\end{equation*}

\subsection{The minor arcs for \texorpdfstring{$\alpha$}{alpha}}

Our first goal is to bound the size of $T(\alpha, \beta)$ on the minor arcs. We want to eventually prove the following:
\begin{lemma}\label{lemma:ch}
    Suppose that \cref{eq:req_temp_1} holds. For every $\theta \in (0, 1]$ and every $\beta \in \T$, we have 
    \begin{enumerate}[label={(\Alph*)}]
        \item the supremum bound \label{ch:sup_bound}
        \begin{equation*}
            \sup_{\alpha \in \mi(\theta)} \abs{T(\alpha, \beta)} \ll X^{s - \frac{t}{\sigma}\theta + \eps},
        \end{equation*}
        and
        \item the mean value bound \label{ch:mv_bound}
        \begin{equation*}
            \oint \int_{\mi(\theta)} \abs{T(\alpha, \beta)} \diff \alpha \diff \beta \ll X^{s-k+d - \frac{t-2 \tau}{\sigma} \theta + \eps}.
        \end{equation*}
    \end{enumerate}
\end{lemma}

\begin{remark}\label{rmk:theta1_bound}
    Although the mean value bound \labelcref{ch:mv_bound} is stated for every $\theta \in (0, 1]$, we use it only for $\theta = 1$. We retain the slightly more general statement since the same proof gives the estimate uniformly for all $\theta$ at no additional cost.
\end{remark}

The proof, being a little lengthy, is relegated to \cref{sec:proof_lemma}.

Once \cref{lemma:ch} has been established, this allows us to bound the integral over the minor arcs. 
\begin{lemma}\label{lemma:minor_mi}
    Assume that \cref{eq:req_temp_1} holds. Let $\theta_* \in (0, 1]$ satisfy
    \begin{equation}\label{eq:theta_star}
        \of{\frac{t}{\sigma} - 2} \theta_* > d.
    \end{equation}
    Then there is a $\nu > 0$ such that for all $\theta \in [\theta_*, 1]$, we have
    \begin{equation*}
         \oint\int_{\mi(\theta)} \abs{T(\alpha, \beta)} \diff \alpha \diff \beta \ll X^{s - k-d - \nu}.
    \end{equation*}
\end{lemma}

\begin{proof}
    As in the proof of \cite[Lemma 3.2]{brandes_hasse_2021}, we use \cref{lemma:ch} and the major arc volume estimate in \cref{eq:major_volume}.

    Construct a decreasing sequence $(\theta_i)$ of the form $1 = \theta_0 > \theta_1 > \dots > \theta_N = \theta_*$. With the requirement for $\theta_*$ in \cref{eq:theta_star}, this can be chosen with $N = O(1)$ and 
    \begin{equation}\label{eq:step_size_theta}
        2\of{\theta_{i-1} - \theta_i} < \of{\frac{t}{\sigma}-2} \theta_* - d. 
    \end{equation}
    Note that
    \begin{equation*}
        \mi(\theta_*) = \mi(\theta_N) = \mi(\theta_0) \cup \bigcup_{i = 1}^N \of{\mi(\theta_i) \setminus \mi( \theta_{i-1})},
    \end{equation*}
    so it suffices to show that $\mi(\theta_0)$ and each set $\mi(\theta_i) \setminus \mi( \theta_{i-1})$ give an acceptable contribution to the integral.

    For $\theta_0 = 1$, we distinguish two different cases. If $k - d \geq 2$, then the mean value bound \labelcref{ch:mv_bound} in \cref{lemma:ch} gives
    \begin{equation*}
        \oint \int_{\mi(1)} \abs{T(\alpha, \beta)} \diff \alpha \diff \beta \ll X^{s -k + d - \frac{t-2 \tau}{\sigma} + \eps },
    \end{equation*}
    which is $\ll X^{s-k-d-\nu}$ for some $\nu > 0$, since $s > s_*$ is equivalent to $(t-2 \tau) / \sigma > 2d$. If $k - d = 1$, then $\sigma = 1$, and the supremum bound \labelcref{ch:sup_bound} of \cref{lemma:ch} gives
    \begin{equation*}
        \oint \int_{\mi(1)} \abs{T(\alpha, \beta)} \diff \alpha \diff \beta \ll \sup_{\substack{\alpha \in \mi(1) \\ \beta \in \T}} \abs{T(\alpha, \beta)} \ll X^{s - t + \eps},
    \end{equation*}
    and this is in turn $\ll X^{s-k-d - \nu}$ for some $\nu > 0$ since $s > s_*$ is equivalent to $t > k+d$.

    Next, for each $i \geq 1$, the supremum bound in \labelcref{ch:sup_bound} together with the volume estimate \cref{eq:major_volume} gives
    \begin{align*}
        \oint \int_{\mi(\theta_i) \setminus \mi(\theta_{i-1})} \abs{T(\alpha, \beta)} \diff \alpha \diff \beta &\ll \operatorname{vol} \Ma(\theta_{i-1}) \sup_{\substack{\alpha \in \mi(\theta_i) \\ \beta \in \T}} \abs{T(\alpha, \beta)} \\
        &\ll X^{-k + 2 \theta_{i-1}} X^{s - \frac{t}{\sigma} \theta_i + \eps}.
    \end{align*}
    Since $t / \sigma > 2$ and $\theta_i \geq \theta_*$, \cref{eq:step_size_theta} shows that
    \begin{equation*}
        2 \theta_{i-1}- \frac{t}{\sigma}\theta_i = 2\of{\theta_{i-1} - \theta_i} -\of{\frac{t}{\sigma}-2}\theta_i \leq 2\of{\theta_{i-1} - \theta_i} -\of{\frac{t}{\sigma}-2}\theta_* < -d,
    \end{equation*}
    which means that each set $\mi(\theta_i) \setminus \mi(\theta_{i-1})$ contributes an acceptable amount. 
\end{proof}

\subsection{Proof of \texorpdfstring{\cref{lemma:ch}}{the lemma}}\label{sec:proof_lemma}

\subsubsection{Change of variables}
By \cref{eq:F_psi}, after relabelling the variables we may write
\begin{equation*}
    F(\xu) = \psi_1(\x^{(1)}) + \dots + \psi_u(\x^{(u)}).
\end{equation*}
Since $F$ is non-singular, so is each $\psi_i$; in particular, $\psi_i \neq 0$, so there is a vector $\w^{(i)} \in \Z^n$ with $\psi_i (\w^{(i)}) \neq 0$. We may assume that $\w^{(i)}$ is primitive, which means there is a matrix $A^{(i)} \in \SL_n(\Z)$ whose first column is $\w^{(i)}$. Let $A \in \SL_s(\Z)$ be the block-diagonal matrix with diagonal blocks $A^{(1)}, \dots, A^{(u)}$, and set
\begin{equation*}
    \psit_i (\x^{(i)}) \eqdef \psi_i (A^{(i)} \x^{(i)}), \quad \Ft(\xu) \eqdef F(A \xu) = \sum_{i = 1}^u \psit_i(\x^{(i)}), \quad \Gt(\xu) \eqdef G(A \xu).
\end{equation*}
Since $\xu \mapsto A \xu$ is a bijection on $\Z^s$, the sum $T(\alpha, \beta)$ is the sum of $e \of{\alpha \Ft(\xu) + \beta \Gt(\xu)}$ over $\xu \in A^{-1}[-X, X]^s \cap \Z^s$. Fix $C \geq 1$ such that $A^{-1}[-X, X]^s \subseteq [-CX, CX]^s$; note that this can be chosen to depend only on the $\psi_i$. Then \cite[Lemma 5.2]{brandes_forms_2014}, with $P = X$ and $CX$ in place of $X$, gives
\begin{equation}\label{eq:T_Ttilde}
    \abs{T(\alpha, \beta)} \ll X^{\eps} \Tt(\alpha, \beta),
\end{equation}
where
\begin{equation*}
    \Tt(\alpha, \beta) \eqdef \sup_{\bm{\eta} \in [0, 1]^s} \abs{\sum_{\xu \in [-CX, CX]^s \cap \Z^s} e \of{\alpha \Ft(\xu) + \beta \Gt(\xu) - \bm{\eta} \cdot \xu}}.
\end{equation*}
That is, at the cost of a factor $X^{\eps}$ and a linear twist, the sum over $A^{-1}[-X, X]^s$ is replaced by a sum over $[-CX, CX]^s$, and, using \cref{eq:T_Ttilde}, it suffices to prove \cref{lemma:ch} with $\Tt(\alpha, \beta)$ in place of $\abs{T(\alpha, \beta)}$. Moreover, the coefficient of $x_1^k$ in $\psit_i$ is
\begin{equation*}
    \psit_i(1, 0, \dots, 0) = \psi_i(\w^{(i)}) \neq 0.
\end{equation*}

For the rest of this proof, we drop the tildes from $\Ft$, $\Gt$ and $\psit_i$, so that each $\psi_i$ has a non-zero coefficient of the monomial $x_1^k$.

\subsubsection{Fixing all but one variable in each block}
For each $i$, write $\x^{(i)} = (z^{(i)}; \y^{(i)})$, with $z^{(i)} \in \Z$ and $\y^{(i)} \in \Z^{n-1}$, and set $\z \eqdef (z^{(1)}, \dots, z^{(u)}) \in \Z^u $ and $\yu \eqdef (\y^{(1)}, \dots, \y^{(u)}) \in \Z^{(n-1)u}$ so that $\xu = (\z; \yu)$. For each $\bm{\eta} \in [0, 1]^s$, the sum inside the supremum defining $\Tt(\alpha, \beta)$ can be written as a sum 
\begin{equation*}
    \sum_{\abs{\yu} \leq CX} U(\yu),
\end{equation*}
where
\begin{equation*}
    U(\yu) = U(\alpha, \beta, \bm{\eta}; \yu) \eqdef \sum_{\abs{z^{(1)}} \leq CX} \cdots \sum_{\abs{z^{(u)}} \leq CX} e\Big(\alpha \sum_{i = 1}^u \psi_i(z^{(i)}; \y^{(i)}) + \beta G\of{\z; \yu} - \bm{\eta} \cdot (\z; \yu)\Big).
\end{equation*}
With this, we can conclude that
\begin{equation}\label{eq:Ttilde_U}
    \Tt(\alpha, \beta) \leq \sup_{\bm{\eta} \in [0, 1]^s} \sum_{\abs{\yu} \leq CX} \abs{U(\alpha, \beta, \bm{\eta}; \yu)}.
\end{equation}
For ease of exposition, set
\begin{equation*}
    f_{\yu}(\z) \eqdef \alpha \sum_{i = 1}^u \psi_i(z^{(i)}; \y^{(i)}) + \beta G(\z; \yu) - \bm{\eta} \cdot (\z; \yu)
\end{equation*}
for each $\z \in \Z^u$.

\subsubsection{Weyl differencing}
For $\h \in \Z^u$ and a function $\phi$ on $\Z^u$, let
\begin{equation*}
    \Delta_{\h} \phi(\z) \eqdef \phi(\z + \h) - \phi(\z)
\end{equation*}
denote the forward differencing operator. For $d$ vectors $\h^{(1)}, \dots, \h^{(d)} \in \Z^u$, let $\h_i \eqdef (h_i^{(1)}, \dots, h_i^{(d)})$ be the vector in $\Z^d$ defined for each $1 \leq i \leq u$ by grouping together the $i$th elements of each vector. Applying Weyl differencing $d$ times as in the proof of \cite[Lemma 2.3]{vaughan_hardy-littlewood_1997}, we find intervals $I_i(\h_i) \subseteq [-CX, CX]$ depending only on $\h_i$ such that
\begin{equation*}
    \abs{U(\yu)}^{2^d} \ll X^{(2^d - d- 1) u}\sum_{\abs{\h^{(1)}} \leq 2CX} \cdots \sum_{\abs{\h^{(d)}} \leq 2CX} \abs{\sum_{z^{(1)} \in I_1(\h_1)} \cdots \sum_{z^{(u)} \in I_u(\h_u)} e \of{\Delta_{\h^{(1)}} \cdots \Delta_{\h^{(d)}}\of{f_{\yu}(\z)}}}.
\end{equation*}
Since $G$ has degree $d$ and $\bm{\eta} \cdot (\z; \yu)$ is linear in $\z$, the terms $\Delta_{\h^{(1)}} \cdots \Delta_{\h^{(d)}} G(\z; \yu)$ and $\Delta_{\h^{(1)}} \cdots \Delta_{\h^{(d)}} (\bm{\eta} \cdot (\z; \yu))$ are independent of $\z$, so that taking absolute values allows us to remove the contributions from $G$ and from the linear twist. Next, note that since the $i$th summand of $F$ depends only on $z^{(i)}$, we have
\begin{equation*}
    \Delta_{\h^{(1)}} \cdots \Delta_{\h^{(d)}} \psi_i(z^{(i)}; \y^{(i)}) = h_i^{(1)} \cdots h_i^{(d)} p_{i, \h_i} (z^{(i)}; \y^{(i)}),
\end{equation*}
where $p_{i, \h_i} (\cdot; \y^{(i)})$ is a polynomial of degree $k-d$ in $z$ with leading coefficient non-zero and independent of $\h_i$. Since $\psi_i$ is homogeneous and $\abs{\h_i}, \abs{\y^{(i)}} \ll X$, the polynomials in the above display satisfy \cref{hyp} with $j = k-d$, uniformly in $\h_i$ and $\y^{(i)}$. Hence
\begin{equation*}
    \abs{U(\yu)}^{2^d} \ll X^{(2^d - d- 1) u}\sum_{\abs{\h^{(1)}} \leq 2CX} \cdots \sum_{\abs{\h^{(d)}} \leq 2CX} \abs{\prod_{i = 1}^u f_i(\alpha; \h_i)},
\end{equation*}
where for $\h = (h^{(1)}, \dots, h^{(d)}) \in \Z^d$, we write (suppressing the dependence on $\y^{(i)}$)
\begin{equation*}
    f_i(\alpha; \h) \eqdef \sum_{z \in I_i(\h)} e \of{\alpha h^{(1)} \cdots h^{(d)}p_{i, \h}(z; \y^{(i)})}.
\end{equation*}

\subsubsection{Separating the blocks}
For $\mathcal{H} \subseteq \Z^d$, set
\begin{equation*}
    g_i(\alpha ; \mathcal{H}) \eqdef \sum_{\h_i \in \mathcal{H}} \abs{f_{i}(\alpha ; \h_i)},
\end{equation*}
and let $\mathcal{H}_0 \eqdef [-2CX, 2CX]^d \cap \Z^d$. Summing over $\h^{(1)}, \dots, \h^{(d)} \in \Z^u$ is equivalent to summing over $\h_1, \dots, \h_u \in \Z^d$, so the right-hand side above factorises. Using the inequality 
\begin{equation*}
    \abs{a_1 \cdots a_m} \leq \abs{a_1}^m + \dots + \abs{a_m}^m \quad \text{for all } a_1, \dots, a_m \in \C,
\end{equation*}
we arrive at
\begin{equation*}
    \abs{U(\yu)}^{2^d} \ll X^{(2^d - d- 1) u} \abs{\prod_{i = 1}^u g_i(\alpha; \mathcal{H}_0)} \ll X^{(2^d - d- 1) u} \sum_{i = 1}^u g_i(\alpha; \mathcal{H}_0)^u.
\end{equation*}
The sum defining $f_i(\alpha; \h)$ is trivial when $h^{(1)} \cdots h^{(d)} = 0$, so we treat these $\h$ separately. Let
\begin{equation*}
    \Zero \eqdef \set{\h \in \mathcal{H}_0 : h^{(1)} \cdots h^{(d)} = 0}, \quad \Zero^c \eqdef \mathcal{H}_0 \setminus \Zero.
\end{equation*}
Since $\# \Zero \ll X^{d-1}$ and $\abs{f_i(\alpha ; \h)} \ll X$ for $\h \in \Zero$, one can bound
\begin{equation*}
    g_i(\alpha ; \mathcal{H}_0) \ll X^d + g_i(\alpha ; \Zero^c),
\end{equation*}
and hence
\begin{equation*}
    \abs{U(\yu)}^{2^d} \ll X^{(2^d - 1) u} + X^{(2^d - d- 1) u} \sum_{i = 1}^u g_i(\alpha; \Zero^c)^u.
\end{equation*}
Taking $2^d$th roots and recalling that $t = u/2^d$, we arrive at 
\begin{equation}\label{eq:U_intermediate}
    \abs{U(\yu)} \ll X^{u -t} + X^{u - (d+1)t} \sum_{i = 1}^u g_i (\alpha; \Zero^c)^t.
\end{equation}
Note that the right-hand side depends neither on $\beta$ nor on $\bm{\eta}$. 

\subsubsection{The supremum bound}
\cref{lemma:2.3} with $j = k-d$ and $r = d$, $p_{\h} = p_{i, \h}(\cdot; \y^{(i)})$ and $\mathcal{D}_{\h} = I_i(\h)$ gives
\begin{equation}\label{eq:g_Z}
    \sup_{\alpha \in \mi(\theta)} g_i(\alpha ; \Zero^c) \ll X^{d + 1 - \frac{\theta}{\sigma} + \eps},
\end{equation}
uniformly in $1 \leq i \leq u$ and $\yu$. Inserting this into \cref{eq:U_intermediate}, we find that 
\begin{equation*}
     \sup_{\alpha \in \mi(\theta)} \abs{U(\yu)} \ll X^{u -t} + X^{u - \frac{\theta t}{\sigma} + t \eps},
\end{equation*}
and recalling \cref{eq:Ttilde_U} gives
\begin{equation*}
    \sup_{\alpha \in \mi(\theta)} \Tt(\alpha, \beta) \ll X^{nu - \frac{\theta t}{\sigma} + t \eps} = X^{s - \frac{\theta t}{\sigma} + t \eps}.
\end{equation*}
With \cref{eq:T_Ttilde}, this gives \labelcref{ch:sup_bound}.

\subsubsection{The mean value bound}
Inserting \cref{eq:U_intermediate} into \cref{eq:Ttilde_U} and integrating, we arrive at
\begin{equation*}
    \oint \int_{\mi(\theta)} \Tt(\alpha, \beta) \diff \alpha \diff \beta \ll X^{s-t} + X^{s - (d+1)t} \max_{i, \yu} \int_{\mi(\theta)} g_i(\alpha; \Zero^c)^t \diff \alpha.
\end{equation*}
Since $t \geq 2 \tau$, we may write $t = v + 2 \tau$ with $v \geq 0$, and so
\begin{equation*}
    \int_{\mi(\theta)} g_i(\alpha; \Zero^c)^t \diff \alpha \ll \of{\sup_{\alpha \in \mi(\theta)} g_i(\alpha; \Zero^c) }^v \int_{\mi(\theta)}g_i(\alpha; \Zero^c)^{2 \tau} \diff \alpha.
\end{equation*}
The first term is covered by \cref{eq:g_Z}, namely
\begin{equation*}
    \of{\sup_{\alpha \in \mi(\theta)} g_i(\alpha; \Zero^c)}^{v} \ll X^{v(d+1) - \frac{\theta v}{\sigma} + \eps},
\end{equation*}
and for the second term, Hölder's inequality gives
\begin{equation*}
    g_i(\alpha; \Zero^c)^{2 \tau} \ll (X^d)^{2 \tau-1} \sum_{\h \in \Zero^c} \abs{f_i(\alpha; \h)}^{2 \tau},
\end{equation*}
so that the mean value estimate in \cref{thm:2.1} with $m = 1$ and $j = k-d$ gives
\begin{equation*}
    \int_{\mi(\theta)} g_i(\alpha; \Zero^c)^{2 \tau} \diff \alpha \ll X^{(2 \tau-1) d} X^d X^{2 \tau - (k-d) + \eps} \ll X^{2 \tau (d+1)- (k-d) + \eps}.
\end{equation*}
Combining these yields
\begin{align*}
    \int_{\mi(\theta)} g_i(\alpha; \Zero^c)^t \diff \alpha
    & \ll X^{v(d+1) - \frac{\theta v}{\sigma} + \eps} X^{2 \tau (d+1)- (k-d) + \eps} \\
    & \ll X^{t(d+1) -(k-d) - \frac{\theta (t - 2 \tau)}{\sigma} + \eps},
\end{align*}
and hence
\begin{equation*}
    \oint \int_{\mi(\theta)} \Tt(\alpha, \beta) \diff \alpha \diff \beta \ll X^{s-t} + X^{s - k +d - \frac{\theta (t - 2 \tau)}{\sigma} + \eps},
\end{equation*}
which together with \cref{eq:T_Ttilde} proves \labelcref{ch:mv_bound}.

\subsection{Two-dimensional split}\label{sec:minor_split}
We have now shown that unless $\alpha$ is well approximated by rationals, the contribution to the minor arc integral is small. As in \cite{browning_forms_2017}, a similar Weyl argument shows that whenever $\alpha$ has a good rational approximation $a/q$, unless $\beta$ is well approximated with a denominator of the form $qr$, the corresponding contribution will also be small. The following lemma makes this precise.

\begin{lemma}\label{lemma:subdivision}
    Let $\kappa > 0$ and $\theta \in (0, 1)$. Suppose that $\alpha \in \Ma(\theta)$ with corresponding rational approximation $a/q$. For any $\eta$ in the range
    \begin{equation}\label{eq:range_eta}
        0 < \eta \leq 1 - \theta,
    \end{equation}
    one of the following three statements is true:
    \begin{enumerate}[label={(\roman*)}]
        \item Either $T$ is small: \label{opt:1}
        \begin{equation*}
            \abs{T(\alpha, \beta)} \ll X^{s- \kappa \eta + \eps},
        \end{equation*}
        \item or there is a natural number $r \leq X^{(d-1)\eta}$ such that \label{opt:2}
        \begin{equation*}
            \norm{q r \beta} \ll X^{-d + (d-1)\eta + \theta},
        \end{equation*}
        \item or $s$ is small: \label{opt:3}
        \begin{equation*}
            s \leq 2^{d-1} \kappa.
        \end{equation*}
    \end{enumerate}
\end{lemma}

\begin{proof}
    This is \cite[Lemma 6.1]{browning_forms_2017}.
\end{proof}

We assume from now on that
\begin{equation*}
    s > 2^{d-1} \kappa,
\end{equation*}
in order to be able to exclude the case \labelcref{opt:3}. By combining the bound on $T(\alpha, \beta)$ in \labelcref{opt:1} with the supremum bound \labelcref{ch:sup_bound} in \cref{lemma:ch}, one can see that it makes sense to choose $\eta$ and $\theta$ to be related by
\begin{equation}\label{eq:kappa_eta_theta}
    \kappa \eta = \frac{t}{\sigma} \theta.
\end{equation}

Motivated by \cref{lemma:subdivision}, we now define a set $\Na(\eta)$ of two-dimensional major arcs, consisting of those $(\alpha, \beta)$ for which both $\alpha$ and $\beta$ are well approximated in the sense of \labelcref{opt:2}, with complement $\nn(\eta)$. With a suitable choice of the constant $c$ below, every $(\alpha, \beta) \in \nn(\eta)$ with $\alpha \in \Ma(\theta)$ satisfies \labelcref{opt:1}.

\begin{definition}\label{def:najor_arcs}
    For a constant $c > 0$, let $\Na(\eta) = \Na(\eta, \theta)$ be the set of $(\alpha, \beta) \in \T^2$ with the property that 
    \begin{align*}
        \norm{\alpha q} \leq c X^{-k+ \theta}, && \norm{\beta qr} \leq c X^{-d+ (d-1)\eta + \theta}
    \end{align*}
    for some $1 \leq q \leq c X^{\theta}$ and $1 \leq r \leq c X^{(d-1)\eta}$. Define $\nn(\eta) \eqdef \T^2 \setminus \Na(\eta)$. 
\end{definition}

Two remarks are in order:
\begin{enumerate}
    \item Combining the bounds on $T(\alpha, \beta)$ in \labelcref{ch:sup_bound} of \cref{lemma:ch} and in \labelcref{opt:1} using the equality \cref{eq:kappa_eta_theta} gives
    \begin{equation}\label{eq:T_bound_nn}
        \abs{T(\alpha, \beta)} \ll X^{s - \kappa \eta + \eps}
    \end{equation}
    for all $(\alpha, \beta) \in \nn(\eta)$. 
    
    \item Using the equality \cref{eq:kappa_eta_theta} again, one can estimate
    \begin{equation}\label{eq:volume_bound_Na}
        \begin{split}
            \operatorname{vol} \Na(\eta) &\ll \sum_{1 \leq q \leq cX^\theta} \sum_{1 \leq a \leq q} \frac{X^{-k + \theta}}{q} \sum_{1 \leq r \leq c X^{(d-1) \eta}} \sum_{1 \leq b \leq qr} \frac{X^{-d + (d-1)\eta + \theta}}{qr} \\
            & \ll X^{-k - d + 2(d-1) \eta + 3 \theta} \\
            & \ll X^{-k-d + \of{2(d-1) + \frac{3 \kappa \sigma }{t}}\eta}.
        \end{split}
    \end{equation}
\end{enumerate}

\subsection{Two-dimensional minor arcs}\label{sec:minor_minor}
Next, we can apply \cref{lemma:minor_mi} and \cref{lemma:subdivision} together with the estimates in \cref{eq:T_bound_nn} and \cref{eq:volume_bound_Na} to bound the integral over $\nn(\eta)$.

\begin{lemma}\label{lemma:minor_nn}
    Suppose that \cref{eq:req_temp_1} holds, together with 
    \begin{equation}\label{eq:kappa_n_1}
        \frac{d}{\kappa} + \frac{(d+2)\sigma}{t} < 1
    \end{equation}
    and 
    \begin{equation}\label{eq:kappa_n_2}
        \frac{2(d-1)}{\kappa} + \frac{3 \sigma}{t} < 1.
    \end{equation}
    Then for any $\eta$ in the range
    \begin{equation}\label{eq:eta_range_kappa}
        0 < \eta \leq \of{1 + \frac{\sigma \kappa}{t}}^{-1},
    \end{equation}
    the integral over the two-dimensional minor arcs can be bounded by
    \begin{equation*}
        \int_{\nn(\eta)} \abs{T(\alpha, \beta)} \diff \alpha \diff \beta \ll X^{s- k-d - \nu}
    \end{equation*}
    for some $\nu > 0$.
\end{lemma}
\begin{proof}
    This is essentially \cite[Lemma 3.4]{brandes_hasse_2021}, with \cref{lemma:minor_mi} replacing \cite[Lemma 3.2]{brandes_hasse_2021}. The completion of the proof requires a $\theta_*$ satisfying \cref{eq:theta_star} (and hence a suitable $\eta_*$). Since the hypothesis \cref{eq:range_eta} of \cref{lemma:subdivision} is equivalent to \cref{eq:eta_range_kappa} under the assumption \cref{eq:kappa_eta_theta} on $\kappa$, $\eta$ and $\theta$, taking
    \begin{align*}
        \eta_{\max} \eqdef \of{1 + \frac{\sigma \kappa}{t}}^{-1}, && \theta_{\max} \eqdef \frac{\sigma \kappa}{t + \sigma \kappa }, 
    \end{align*}
    gives a permissible $\theta_{\max}$, and the requirement from \cref{eq:theta_star}, namely
    \begin{equation*}
        \theta_{\max} > \frac{d \sigma}{t- 2\sigma},
    \end{equation*}
    is then equivalent to \cref{eq:kappa_n_1}.
\end{proof}

It remains to choose $\kappa$. Since $t = s/(2^d n)$, the conditions \cref{eq:kappa_n_1,eq:kappa_n_2} can be satisfied by some $\kappa$ when $s > s_2$ and $s > s_3$, respectively. Since both hold by \cref{rmk:s2_s3}, we may choose $\kappa < 2^{1-d}s$ to satisfy \cref{eq:kappa_n_1,eq:kappa_n_2}; this also ensures that $s > 2^{d -1} \kappa$. 

To summarise, \cref{lemma:minor_mi} shows that the contribution from the minor arcs $\mi(\theta)$ is suitably small, and \cref{lemma:subdivision} ensures that for any $\eta$, either $(\alpha, \beta)$ lies on the major arcs $\Na(\eta)$, or $T(\alpha, \beta)$ is small (in terms of $\kappa$ and $\eta$). By \cref{lemma:minor_nn}, for small enough $\eta$, the contribution from the two-dimensional minor arcs $\nn(\eta)$ is acceptable. Hence for any sufficiently small $\eta > 0$, we have
\begin{equation}\label{eq:only_major_left}
    N_{F, G}(X) = \int_{\Na(\eta)} T(\alpha, \beta) \diff \alpha \diff \beta + O \of{X^{s - k-d - \nu}}. 
\end{equation}
What remains is then to evaluate the contribution from the major arcs $\Na(\eta)$, which is the content of the next section.

\section{Major arc contribution}\label{sec:major}

The analysis of the major arcs is performed as in \cite[\S 4]{brandes_hasse_2021}. We begin by enlarging the set of major arcs.

\begin{definition}
    Set
    \begin{equation}\label{eq:omega}
        \omega \eqdef (d-1)\eta + \theta.
    \end{equation}
    Let $\Pa(\omega)$ be the set of $(\alpha, \beta)\in \T^2$ with the property that there is a natural number $q \leq c' X^{\omega}$, a pair $(a, b) \in \Z^2$ and a pair $(\gamma, \delta) \in \R^2$ such that
    \begin{equation}\label{eq:alphas_major}
        \alpha = \frac{a}{q} + \gamma \quad \text{and} \quad \beta = \frac{b}{q} + \delta,
    \end{equation}
    with
    \begin{equation*}
        \abs{\gamma} = \abs{\alpha - \frac{a}{q}} \leq c' X^{-k + \omega} \quad \text{and} \quad \abs{\delta} = \abs{\beta - \frac{b}{q}} \leq c' X^{-d + \omega}
    \end{equation*}
    for some suitable constant $c'$.
\end{definition}

Two remarks are again in order:
\begin{enumerate}
    \item Looking at \cref{def:najor_arcs}, it is clear that we can pick $c'$ such that $\Na(\eta) \subseteq \Pa(\omega)$. Then \cref{lemma:minor_nn} (under the hypotheses \cref{eq:req_temp_1,eq:req_temp_2,eq:req_temp_3}) shows that there is a $\nu > 0$ such that
    \begin{equation*}
        \int_{\Pa(\omega) \setminus \Na(\eta)} T(\alpha, \beta) \diff \alpha \diff \beta \ll \int_{\nn(\eta)} \abs{T(\alpha, \beta)} \diff \alpha \diff \beta \ll X^{s-k-d - \nu}.
    \end{equation*}
    \item We can compute directly:
    \begin{equation*}
        \operatorname{vol} \Pa(\omega) \ll \sum_{1 \leq q \leq c' X^{\omega}} \sum_{1 \leq a \leq q} c' X^{-k + \omega} \sum_{1 \leq b \leq q} c' X^{-d + \omega} \ll X^{-k-d + 5 \omega}.
    \end{equation*}
\end{enumerate}

Let 
\begin{equation*}
    S(q; a, b) \eqdef \sum_{\xu \Mod q} e_q \of{a F(\xu) + b G(\xu)}
\end{equation*}
and recall the definition of $J(\gamma, \delta)$ in \cref{eq:J_def}.

To make the notation simpler, set
\begin{equation}\label{eq:J_X}
    J_X(\gamma, \delta) \eqdef X^s J(X^k \gamma, X^d \delta) = \int_{[-X, X]^s} e \of{\gamma F(\bm{\xi}) + \delta G(\bm{\xi})} \diff \bm{\xi}.
\end{equation}

\begin{lemma}\label{lemma:rewrite_major_int_1}
    The integral over the major arcs can be written as
    \begin{equation*}
    \begin{split}
        \int_{\Pa(\omega)} T(\alpha, \beta) \diff \alpha \diff \beta & = \sum_{q \leq c' X^{\omega}} q^{-s} \sum_{\substack{1 \leq a, b \leq q \\ (q, a, b) = 1}} S(q; a, b) \int_{\substack{\abs{\gamma} \leq c' X^{-k + \omega} \\ \abs{\delta} \leq c' X^{-d + \omega}}} J_{X} (\gamma, \delta) \diff \gamma \diff \delta 
        \\
        & \quad{} + O \of{X^{s-k-d-1 + 7 \omega}}.
    \end{split}
    \end{equation*}
\end{lemma}
\begin{proof}
    For each $(\alpha, \beta) \in \Pa(\omega)$, writing $\alpha$ and $\beta$ as in \cref{eq:alphas_major} gives
    \begin{equation}\label{eq:rewrite_S_q_J}
        \abs{T(\alpha, \beta) - q^{-s} S(q; a, b) J_{X}(\gamma, \delta)} \ll X^{s-1} q \of{1 + X^k \abs{\gamma} + X^d \abs{\delta}},
    \end{equation}
    and the right-hand side is $\ll X^{s-1+ 2 \omega}$ for $(\alpha, \beta) \in \Pa(\omega)$ by construction.
\end{proof}

\begin{definition}
    For a large number $R$, define the truncated singular series
    \begin{equation*}
        \Series(R) \eqdef \sum_{1 \leq q \leq R} q^{-s} \sum_{\substack{1 \leq a, b \leq q \\ (q, a, b) = 1}} S(q; a, b),
    \end{equation*}
   and the truncated singular integral
    \begin{equation*}
        \Jintegral(R) \eqdef \int_{\substack{\abs{\gamma} \leq R \\ \abs{\delta} \leq R}} J(\gamma, \delta) \diff \gamma \diff \delta.
    \end{equation*}
\end{definition}

A change of variables $\gamma' = X^k \gamma$ and $\delta' = X^d \delta$ in \cref{lemma:rewrite_major_int_1}, combined with \cref{eq:J_X}, gives
\begin{equation}\label{eq:rewrite_major_int_2}
    \int_{\Pa(\omega)} T(\alpha, \beta) \diff \alpha \diff \beta = X^{s - k-d} \Series(c'X^\omega) \Jintegral(c'X^{\omega}) + O \of{X^{s - k-d-1 + 7 \omega}}.
\end{equation}

\begin{lemma}\label{lemma:J}
    Let $J(\gamma, \delta)$ and $\Jintegral(R)$ be as above.
    \begin{enumerate}[label={(\roman*)}]
        \item For any $(\gamma, \delta) \in \R^2$, we have \label{item:J_bound}
        \begin{equation*}
            \abs{J(\gamma, \delta)} \ll \min \set{1, \abs{\gamma}^{-\frac{t}{\sigma} + \eps}, \abs{\delta}^{- \of{\frac{d-1}{\kappa} + \frac{\sigma}{t}}^{-1} + \eps}}.
        \end{equation*}
        \item Suppose that \cref{eq:kappa_n_2} holds. Then the limit $\Jintegral\eqdef \lim_{R \to \infty} \Jintegral(R)$ exists, and there is a $\nu > 0$ such that \label{item:J_conv}
        \begin{equation*}
            \abs{\Jintegral(2R) - \Jintegral(R)} \ll R^{- \nu}. 
        \end{equation*}
    \end{enumerate}
\end{lemma}
\begin{proof}
    This is:
    \begin{enumerate}[label={(\roman*)}]
        \item \cite[Lemma 4.1]{brandes_hasse_2021} with $\rho = 1$, and 
        \item the analogue of \cite[Lemma 4.2]{brandes_hasse_2021} using \labelcref{item:J_bound}.
    \end{enumerate}
\end{proof}

\begin{lemma}\label{lemma:S}
    Let $S(q; a, b)$ and $\Series(R)$ be as above.
    \begin{enumerate}[label={(\roman*)}]
        \item For any $q \in \N$ and any $(a, b) \in (\Z/q\Z)^2$ with $(q, a, b) = 1$, we have \label{item:S_bound}
        \begin{equation*}
            q^{-s} \abs{S(q; a, b)} \ll q^{\eps} \min \set{\of{\frac{q}{(q, a)}}^{-\frac{t}{\sigma}}, q^{- \of{\frac{d-1}{\kappa} + \frac{\sigma}{t}}^{-1}}}.
        \end{equation*}
        \item Suppose that \cref{eq:kappa_n_2} holds. Then the limit $\Series \eqdef \lim_{R \to \infty} \Series(R)$ exists, and there is a $\nu > 0$ such that \label{item:S_conv}
        \begin{equation*}
            \abs{\Series(2R) - \Series(R)} \ll R^{- \nu}. 
        \end{equation*}
    \end{enumerate}
\end{lemma}
\begin{proof}
    This is:
    \begin{enumerate}[label={(\roman*)}]
        \item \cite[Lemma 4.3]{brandes_hasse_2021} with $\rho = 1$, and
        \item the analogue of \cite[Lemma 4.4]{brandes_hasse_2021} using \labelcref{item:S_bound}.
    \end{enumerate}
\end{proof}

\section{Proof of \texorpdfstring{\cref{thm:main}}{the theorem}}\label{sec:proof}
\begin{proof}[Proof of \cref{thm:main}]
    Assume that $s > s_*$. After \cref{eq:only_major_left}, the only task remaining is to evaluate the integral over $\Na(\eta)$ for some sufficiently small $\eta$. Using \cref{eq:rewrite_major_int_2}, we can write the integral over the enlarged major arcs in terms of $\Series(c'X^\omega)$ and $\Jintegral(c'X^{\omega})$. Since \cref{lemma:J,lemma:S} guarantee the convergence that we need, and since the restrictions on $s$ imply that the parameter $\kappa$ can be chosen to satisfy \cref{eq:kappa_n_2}, we get
    \begin{equation*}
        N_{F, G}(X) = X^{s - k-d } \Jintegral \Series + O \of{X^{s - k-d - \nu}}
    \end{equation*}
    for some $\nu > 0$. 

    It remains to identify the constant and discuss its positivity. Since $\Jintegral(R)$ is the integral of $J$ over the box $[-R, R]^2$, we have $\Jintegral = \chi_{\infty}$ by definition. Next, the arguments of \cite[Theorem 2.4]{vaughan_hardy-littlewood_1997} show that $\Series = \prod_p \chi_p$, and that there is a positive number $p_0$ such that
    \begin{equation*}
        \frac{1}{2} < \prod_{p > p_0} \chi_p < \frac{3}{2}.
    \end{equation*}
    Since $F$ and $G$ are homogeneous, if $\xu$ is a non-singular solution of the system, then so is $\lambda \xu$ for every $\lambda \neq 0$, so that a non-singular solution in $\Q_p^s$ yields one in $\Z_p^s$, and one in $\R^s$ similarly yields one in $[-1, 1]^s$. If the system has a non-singular solution in $\Z_p^s$, then Hensel's lemma shows that $\Gamma(p^h) \gg p^{h(s-2)}$ for all large $h$, whence $\chi_p > 0$. If this holds for every prime $p$, then the above bound for the tail of the product gives $\Series > 0$. Likewise, the arguments of \cite[Lemma 2 and \S 11]{schmidt_simultaneous_1982} show that $\Jintegral > 0$ if the system has a non-singular solution in $[-1, 1]^s$.
\end{proof}

\printbibliography

\end{document}

%% file: preamble/packages.tex
\documentclass[a4paper]{amsart}
\usepackage[top=3cm, bottom=3cm,inner=3cm, outer=3cm]{geometry}	
\usepackage[british]{babel} 
\usepackage{csquotes}

\usepackage{xcolor} 
\usepackage{graphicx}
\usepackage{amsmath, amssymb}
\usepackage{amsfonts} 
\usepackage{amsthm} 
\usepackage{hyperref}
\usepackage[capitalise]{cleveref}
\usepackage{bbm,bm}

\newtheorem{theorem}{Theorem}[section] 
\newtheorem{lemma}[theorem]{Lemma} 
\newtheorem{corollary}[theorem]{Corollary}

\theoremstyle{definition}
\newtheorem{definition}{Definition}[section]
\newtheorem{remark}[theorem]{Remark}
\numberwithin{equation}{section}
\newtheorem{hypothesis}{Hypothesis}
\crefname{hypothesis}{Hypothesis}{Hypotheses}

\crefname{equation}{}{}
\creflabelformat{equation}{\textup{(#2#1#3)}}

\usepackage{enumitem}
\setlist{topsep=4pt,itemsep=1pt,partopsep=1pt,parsep=2pt}

%% file: preamble/macros.tex
\newcommand*\diff{\mathop{}\!\mathrm{d}}
\newcommand{\Mod}[1]{\, (\mathrm{mod} \, #1)}
\newcommand{\eqdef}{: =}

\newcommand{\SL}{\mathrm{SL}}

\newcommand{\C}{\mathbb{C}}
\newcommand{\Constant}{\mathcal{C}}

\newcommand{\Ft}{\tilde{F}}

\newcommand{\Gt}{\tilde{G}}

\newcommand{\h}{\mathbf{h}}

\newcommand{\Jintegral}{\mathfrak{J}}

\newcommand{\mi}{\mathfrak{m}}
\newcommand{\Ma}{\mathfrak{M}}

\newcommand{\N}{\mathbb{N}}
\newcommand{\nn}{\mathfrak{n}}
\newcommand{\Na}{\mathfrak{N}}

\newcommand{\Pa}{\mathfrak{P}}

\newcommand{\psit}{\tilde{\psi}}

\newcommand{\Q}{\mathbb{Q}}

\newcommand{\R}{\mathbb{R}}

\newcommand{\Series}{\mathfrak{S}}

\newcommand{\T}{\mathbb{T}}
\newcommand{\Tt}{\tilde{T}}

\newcommand{\w}{\mathbf{w}}

\newcommand{\x}{\mathbf{x}}

\newcommand{\xu}{\underline{\x}}

\newcommand{\y}{\mathbf{y}}
\newcommand{\yu}{\underline{\y}}

\newcommand{\z}{\mathbf{z}}
\newcommand{\Z}{\mathbb{Z}}
\newcommand{\Zero}{\mathcal{Z}}

\let\eps\relax
\newcommand{\eps}{\varepsilon}

\newcommand{\norm}[1]{\left\| #1 \right\|}

\newcommand{\set}[1]{\left\{ #1 \right\}}
\newcommand{\of}[1]{\left( #1 \right)}

\newcommand{\abs}[1]{\left| #1 \right|}
\newcommand{\floor}[1]{\left\lfloor #1 \right\rfloor}

%% file: bibliography.bib
@article{wooley_nested_2019,
	title = {Nested efficient congruencing and relatives of {Vinogradov}'s mean value theorem},
	volume = {118},
	doi = {10.1112/plms.12204},
	number = {4},
	journal = {Proceedings of the London Mathematical Society. Third Series},
	author = {Wooley, T. D.},
	year = {2019},
	mrnumber = {3938716},
	pages = {942--1016},
}

@article{wooley_rational_2015,
	title = {Rational solutions of pairs of diagonal equations, one cubic and one quadratic},
	volume = {110},
	copyright = {http://doi.wiley.com/10.1002/tdm\_license\_1.1},
	doi = {10.1112/plms/pdu054},
	number = {2},
	journal = {Proceedings of the London Mathematical Society},
	author = {Wooley, T. D.},
	year = {2015},
	pages = {325--356},
}

@book{vaughan_hardy-littlewood_1997,
	address = {Cambridge},
	edition = {2},
	series = {Cambridge {Tracts} {In} {Mathematics}},
	title = {The {Hardy}-{Littlewood} method},
	isbn = {978-0-521-57347-4},
	number = {125},
	publisher = {Cambridge Univ. Press},
	author = {Vaughan, R. C.},
	year = {1997},
}

@incollection{schmidt_simultaneous_1982,
	series = {Progress in {Mathematics}},
	title = {Simultaneous rational zeros of quadratic forms},
	volume = {22},
	booktitle = {Séminaire {Delange}-{Pisot}-{Poitou} ({Théorie} des {Nombres}), {Paris} 1980--1981},
	publisher = {Birkhäuser},
	author = {Schmidt, W. M.},
	year = {1982},
	pages = {281--307},
}

@article{browning_forms_2017,
	title = {Forms in many variables and differing degrees},
	volume = {19},
	doi = {10.4171/jems/668},
	number = {2},
	journal = {Journal of the European Mathematical Society},
	author = {Browning, T. D. and Heath-Brown, D. R.},
	year = {2017},
	pages = {357--394},
}

@article{browning_rational_2015,
	title = {Rational points on intersections of cubic and quadric hypersurfaces},
	volume = {14},
	doi = {10.1017/S1474748014000127},
	number = {4},
	journal = {Journal of the Institute of Mathematics of Jussieu},
	author = {Browning, T. D. and Dietmann, R. and Heath-Brown, D. R.},
	year = {2015},
	pages = {703--749},
}

@article{brandes_hasse_2021,
	title = {The {Hasse} principle for diagonal forms restricted to lower-degree hypersurfaces},
	volume = {15},
	doi = {10.2140/ant.2021.15.2289},
	number = {9},
	journal = {Algebra \& Number Theory},
	author = {Brandes, J. and Parsell, S. T.},
	year = {2021},
	note = {Publisher: MSP},
	pages = {2289--2314},
}

@unpublished{brandes_parsell_corrigendum,
  author   = {Brandes, Julia and Parsell, Scott T.},
  title    = {Corrigendum: The {H}asse principle for diagonal forms restricted to lower-degree hypersurfaces},
  note     = {To appear in Algebra \& Number Theory}
}

@misc{wooley_main_2026,
	title = {The main conjecture for a {Vinogradov} subsystem},
	doi = {10.48550/arXiv.2610.05515},
	publisher = {arXiv},
	author = {Wooley, T. D.},
	year = {2026},
	note = {arXiv:2610.05515 [math.NT]},
}

@misc{theorin_johansson_hasse_2026,
	title = {The {Hasse} principle for diagonal forms restricted to a hypersurface of adjacent degree},
	doi = {10.48550/arXiv.2601.18581},
	publisher = {arXiv},
	author = {Theorin Johansson, A.},
	year = {2026},
	note = {arXiv:2601.18581 [math.NT]},
}

@article{wooley_weyls_1999,
	title = {On {Weyl}'s inequality, {Hua}'s lemma, and exponential sums over binary forms},
	volume = {100},
	doi = {10.1215/S0012-7094-99-10014-7},
	number = {3},
	journal = {Duke Mathematical Journal},
	author = {Wooley, T. D.},
	year = {1999},
	pages = {373--423},
}

@article{brandes_forms_2014,
	title = {Forms representing forms and linear spaces on hypersurfaces},
	volume = {108},
	doi = {10.1112/plms/pdt044},
	number = {4},
	journal = {Proceedings of the London Mathematical Society},
	author = {Brandes, J.},
	year = {2014},
	pages = {809--835},
}
